\documentclass[12pt]{amsart}
\usepackage{amsmath, amssymb, eucal, amscd, amstext}

\usepackage{enumerate}

\newtheorem{theorem}{Theorem}[section]
\newtheorem{corollary}[theorem]{Corollary}
\newtheorem{lemma}[theorem]{Lemma}

\theoremstyle{definition}

\theoremstyle{remark}
\newtheorem{remark}[theorem]{Remark}

\newtheorem{example}[theorem]{Example}
\numberwithin{equation}{section}

\newcommand{\smnoind}{\smallskip\noindent}

\newcommand{\CK}{\mathcal{K}}

\newcommand{\cv}{\mathbf{c}}

\begin{document}

\baselineskip=17pt

\title[Linear orthogonality preservers of Hilbert $C^*$-modules: General case]{Linear orthogonality
preservers of Hilbert $C^*$-modules over general $C^*$-algebras}

\author{Chi-Wai Leung, Chi-Keung Ng  \and Ngai-Ching Wong }

\address[Chi-Wai Leung]{Department of Mathematics, The Chinese
University of Hong Kong, Hong Kong.}
\email{cwleung@math.cuhk.edu.hk}

\address[Chi-Keung Ng]{Chern Institute of Mathematics and LPMC, Nankai University, Tianjin 300071, China.}
\email{ckng@nankai.edu.cn}

\address[Ngai-Ching Wong]{Department of Applied Mathematics, National Sun Yat-sen University,  Kaohsiung, 80424, Taiwan, R.O.C.}
\email{wong@math.nsysu.edu.tw}

\thanks{The authors are supported by Hong Kong RGC Research Grant (2160255),
National Natural Science Foundation of China (10771106), and
Taiwan NSC grant (NSC96-2115-M-110-004-MY3).}

\date{\today}

\keywords{}

\subjclass[2000]{46L08, 46H40}

\begin{abstract}
As a partial generalisation of the Uhlhorn theorem to Hilbert $C^*$-modules, we show in this article that the module structure and the
orthogonality structure of a Hilbert $C^*$-module determine its
Hilbert $C^*$-module structure.
In fact, we have a more general result as follows.
Let $A$ be a $C^*$-algebra, $E$ and $F$ be Hilbert $A$-modules, and $I_E$ be the ideal of $A$ generated by $\{\langle x,y\rangle_A: x,y\in E\}$.
If $\Phi : E\to F$ is an $A$-module map, not assumed to be bounded but satisfying
$$
\langle \Phi(x),\Phi(y)\rangle_A\ =\ 0\quad\text{whenever}\quad\langle x,y\rangle_A\ =\ 0,
$$
then there exists a
unique central positive multiplier $u\in M(I_E)$ such that
$$
\langle \Phi(x), \Phi(y)\rangle_A\ =\ u
\langle x, y\rangle_A\qquad (x,y\in E).
$$
As a consequence, $\Phi$ is automatically bounded, the induced map $\Phi_0: E\to \overline{\Phi(E)}$ is adjointable, and $\overline{Eu^{1/2}}$ is isomorphic to $\overline{\Phi(E)}$ as Hilbert $A$-modules.
If, in addition, $\Phi$ is bijective, then $E$ is isomorphic to $F$. 
\end{abstract}

\maketitle

\section{Introduction}

\medskip

It is well known that the norm and the inner product of a Hilbert space $H$ determine
each other, through a polarization formula. 
By the Uhlhorn theorem (which generalized the famous Wigner theorem; see e.g.\ \cite{Molnar2}), the orthogonality structure of the projective space (i.e.\ collection of $\mathbb{C}$-rays) of $H$ determines its real Hilbert space structure if $\dim H \geq 3$ (see e.g.\ \cite[2.2.2]{Molnar}). 
In the case when the linear structure of the Hilbert space is also considered, one can relax the two-way orthogonality preserving assumption in the Uhlhorn theorem and obtain the following result.  
\begin{quotation}
{\it If $\theta$ is a bijective linear map between Hilbert spaces satisfying $\langle\theta(x), \theta(x)\rangle =0$ whenever $\langle x, y \rangle = 0$, then $\theta$ is a scalar
multiple of a unitary.}
\end{quotation}

\medskip

Recall that a Banach space $E$ is called a \emph{Hilbert $A$-module} (where $A$ is a $C^*$-algebra) if $E$ is a right $A$-module $E$ equipped with a positive definite Hermitian $A$-form $\langle\cdot, \cdot\rangle_A$ on $E\times E$ such that the norm of any $x\in E$ coincides with $\|\langle x,x \rangle_A\|^{1/2}$. 
In \cite{Frank-webpage},  M. Frank stated, as one of  the four major open problems in Hilbert C*-module theory,
whether  the above statement is true for general Hilbert $C^*$-modules.
The exact form of his question is as follow. 
\begin{quotation}
Prove or disprove: Each injective bounded
$C^*$-linear orthogonality-preserving mapping $T$
on a Hilbert $C^*$-module over a given $C^*$-algebra $A$ is of the form $T= tU$ for some $C^*$-linear isometric mapping $U$ on the Hilbert $C^*$-module and for some element $t$ of the center $Z(M(A))$ of the multiplier $C^*$-algebra of $A$  which does not admit zero divisors therein.
\end{quotation}

In the case when $T$ is bijective, one may regard a positive answer to the above question as a generalization of a Uhlhorn type theorem to Hilbert $A$-modules, where only one-way orthogonality preserving property is assumed but the linear structure is also considered. 
Notice that one is almost forced to take into account of the $A$-module structure because the question will not have a positive answer if one considers orthogonality preserving map on $\mathbb{C}$-rays (see the example concerning $\bar H$ below) and it is not clear how to give a natural notion of ``$A$-rays''. 
On the other hand, as Hilbert $C^*$-modules are important objects (see e.g.\ \cite{Lance}) because they are the main ingredients in the theory of Strong Morita equivalences (see e.g.\ \cite{Rie}), KK-theory (see e.g.\ \cite{bkd}) and $C^*$-correspondences (see e.g.\ \cite{Kat}), it is thus potentially useful if one can recover the structure of a Hilbert $C^*$-module from some partial information about it.

\medskip

The aim of this article is to investigate the above question of Frank. 
Strictly speaking, this question has a negative answer (see Example \ref{eg-I}(a)). 
However, in Corollary \ref{cor:adj-inj}(a), we show that one can get a positive answer to this question if one slightly changes the expected conclusion (note that in this case, neither the injectivity nor the continuity of the given orthogonality preserving map is necessary). 
Our result can be formulated as follows: 
\begin{quotation}
{\it
Let $A$ be a $C^*$-algebra, and let $E$ and $F$ be Hilbert
$A$-modules.
Suppose that $T: E\rightarrow F$ is an $A$-module map, which is not assumed to be bounded, but is
\emph{orthogonality preserving}, in the sense that for any $x,y\in E$,
$$\langle x,y\rangle_A = 0\quad \Rightarrow\quad \langle T(x),T(y)\rangle_A = 0.$$
There exist a positive element $t$ in the center of the multiplier algebra of the closed linear span, $I_E$, of $\{\langle x, y \rangle_A: x,y\in E\}$ as well as a Hilbert $A$-module isomorphism $U: \overline{Et} \to \overline{\Phi(E)}$ such that $T(x) = U(xt)$ for any $x\in E$ (see Remark \ref{rem:rt-mult}(b) for the meaning of $xt$).} 
\end{quotation}
In the case when $T$ is bijective, we obtain in Theorem \ref{thm:I+mod+op=isom}(b), an analogue of the displayed statement in the first paragraph:
\begin{quotation}
{\it
Let $A, E$ and $F$ be as in the above.
If $T: E\rightarrow F$ is a bijective orthogonality preserving $A$-module map (not assumed to be bounded), there exists an invertible element $t\in Z(M(I_E))_+$ such that $x\mapsto T(xt^{-1})$ is a Hilbert $A$-module isomorphism from $E$ onto $F$.}
\end{quotation}
This result implies that the $A$-module structure and the orthogonality structure of $E$ determine the Hilbert $A$-module $E$ up to a Hilbert $A$-module automorphism.

\medskip

We remark that this positive answer is somewhat surprising because the
orthogonal structure of a general Hilbert $C^*$-module is not as
rich as that of a Hilbert space. For example, the conjugate Hilbert
space $\bar H$ of a complex Hilbert space $H$ can be  regarded as a Hilbert
$\mathcal{K}(H)$-module (where $\CK(H)$ is the $C^*$-algebra of all
compact operators on $H$), and for any $\bar x, \bar y\in \bar H$,
one has $\langle \bar x, \bar y\rangle_{\CK(H)} = 0$ if and only if
either $\bar x=0$ or $\bar y =0$ (recall that $\langle \bar x, \bar
y\rangle_{\CK(H)}(z) = y \langle x, z \rangle$ for any $z\in H$).
This simple example also tells us that the above result will not be true if $T$ is only a $\mathbb{C}$-linear map instead of an $A$-module map. 

\medskip

In order to obtain Corollary \ref{cor:adj-inj} and Theorem \ref{thm:I+mod+op=isom}, we need Theorem \ref{thm-unbdd}, which says that if $\Phi$ is an orthogonality preserving $A$-module map (not necessarily bijective), one can find a (unique) element $u\in Z(M(I_E))_+$ such that
$$
\langle \Phi(x), \Phi(y) \rangle_A\ =\ u \langle x, y \rangle_A 
\quad  (x,y\in E).
$$
In the case when $A$ is a standard $C^*$-algebra, this result was
established in \cite{Turnsek-JMAA}.
In the case when $A$ is
commutative and $E$ is full (i.e.\ $I_E = A$), the above result can be found in \cite{LNW-orth-pres-comm}.
Moreover, in \cite{LNW-orth-pres-rr0}, we proved this result in the case when $A$ has real rank zero and $E$ is full.
On the other hand, it was consider in \cite{FMP} the above result in the case when one adds the assumptions that $\Phi$ is bounded and $E$ is full.
It happens that the idea of the proofs in these papers are very different, and none
of  them seem to be suitable for the general case.
In fact, our proof employs techniques concerning open projections.

\medskip

On the other hand, since $E$ and $F$ can be embedded into their
respective linking algebras, some readers may consider the
possibility of extending the orthogonality preserving map $\Phi$ to a disjointness preserver between the linking algebras, and using the
corresponding results for disjointness preservers in, e.g., \cite{CKLW03, KLW04, LW-ZPP,
Wong-FunctionSpaces},  to obtain Theorem \ref{thm-unbdd}. 
However, if one
wants to do this, the first difficulty is whether there is a
canonical map from $\CK(E)$ to $\CK(F)$ that is compatible with
$\Phi$ (notice that $\Phi$ is not even assumed to be bounded).
Nevertheless, after obtaining Theorem \ref{thm-unbdd},
we can use it to show that such an extension is possible (see Theorem \ref{thm:link-alg}), but we do not
see any easy way to obtain it without our main theorems.
Note also that Theorem \ref{thm:link-alg} can be regarded as an extension of Theorem \ref{thm-unbdd}.

\medskip

Let us mention here that, unlike the situation in some other
literature (e.g.\ \cite{FMP}), $\Phi$ is not assumed to be bounded.
This is because of the philosophy as stated in the first paragraph.

\medskip
Our final remark in this section is about a related work of J. Schweizer.
Recall that for a Hilbert $A$-module $X$, the $C^*$-algebra generated by elementary
operators $\theta_{\zeta,\eta}(\xi) := \eta\langle \zeta, \xi\rangle_A$ ($\zeta,\eta,\xi\in X$) is denoted by $\mathcal{K}(X)$.
In this way,
$X$ becomes a Hilbert $\mathcal{K}(X)$-$A$-bimodule.
Schweizer showed in his PhD thesis
(see \cite[9.6]{Sch96}) that if $T$ is a \emph{bounded} orthogonality preserving $\mathbb C$-linear map
from a full Hilbert $C$-module $X$ into a full Hilbert $D$-module $Y$ (where $C$ and $D$ are $C^*$-algebras), then
there is a ``local morphism'' $\pi:\mathcal{K}(X)\to \mathcal{K}(Y)$ such that
$$
T(ax) = \pi(a)T(x) \quad (a\in \CK(X)),
$$
or equivalently,
$$
\theta_{T\zeta,T\zeta} \leq \|T\|^2\,\pi(\theta_{\zeta,\zeta}) \quad (\zeta\in X).
$$
One may speculate whether this result of Schweizer have some overlap with our main theorems.
However, it is not the case.
For instant, if $H$ is a complex Hilbert space
and $X := \bar H$ is regarded the full Hilbert $\CK(H)$-module as before, then $\CK(X) = \mathbb{C}$ and \cite[9.6]{Sch96} gives us merely
the trivial conclusion that a bounded orthogonality preserving $\mathbb C$-linear map
$T: X \to X$ is $\mathbb{C}$-linear.
Our main theorem, however, implies that any orthogonality preserving $\CK(H)$-module map $T:X\to X$ is a scalar multiple of an isometry. 
Therefore, Schweizer's result does not seem to shed any light on the proof of the main theorems in this article.

\medskip

{\it Acknowledgement.}
We appreciate M. Frank for sending us his recent preprint \cite{FMP}, in which
the case of bounded orthogonality preserving $A$-module maps is considered, through a quite different and independent approach. 

\medskip

\section{Notation and Preliminary}

\medskip

Let us first set some notations.
Throughout this article, $A$ is a $C^*$-algebra and 
$A^{**}$ is the bidual of $A$ (which is a von Neumann algebra).
We denote by $Z(A)$ and $M(A)$ respectively, the center and the space of all
multipliers of $A$.
Moreover, ${\rm Proj}_1(A)$ is the collection of all non-zero
projections in $A$.
Note that if $p\in {\rm
Proj}_1(A^{**})$ is an open projection, then the $C^*$-subalgebra $A \cap pA^{**}p$ is weak-*-dense in $pA^{**}p$ (see e.g.\ \cite{Bro88} or \cite[3.11.9]{Ped}).

\medskip

If $a\in A_+$, we consider $C^*(a)$ to be the $C^*$-subalgebra generated by
$a$, and let $\cv(a)$ be the central cover of $a$ in $A^{**}$ (see
e.g.\ \cite[2.6.2]{Ped}).
If $\alpha, \beta\in \mathbb{R}_+$, we set
$e_a(\alpha, \beta)$ and $e_a(\alpha, \beta]$ to be the spectral
projections (in $A^{**}$) of $a$ corresponding
respectively, to the sets $(\alpha,
\beta)\cap \sigma(a)$ and $(\alpha, \beta]\cap \sigma(a)$.
When $\{a_\lambda\}$ is an increasing net (respectively, a
decreasing net) in $A^{**}_{sa}$, the notation $a_\lambda
\uparrow a$ (respectively, $a_\gamma \downarrow a$) means that
$a_\lambda \to a$ in the weak-*-topology.

\medskip

On the other hand, throughout this article, $E$ and $F$ are non-zero Hilbert
$A$-modules. Unless specified otherwise, $\Phi:E\to F$ is an
orthogonality preserving (see the above for its meaning) $A$-module map,
which is not assumed to be bounded. For simplicity, we write
$\langle x,y\rangle$ instead of $\langle x,y\rangle_A$ when both $x$
and $y$ are in $E$ (or $F$).
Recall that $E$ is said to be \emph{full} if $I_E = A$ (where $I_E$ is as in the above).
For
any $C^*$-subalgebra $B\subseteq A$, we put $E\cdot B := \{xb:x\in
E;b\in B\}$.
By the Cohen Factorisation theorem, $E\cdot B$ coincides with its norm closed linear span. 

\medskip

We now recall the following elementary result (see e.g.\ \cite{LNW-orth-pres-rr0}).

\medskip

\begin{lemma}\label{lem:cp}
Suppose that $p\in {\rm Proj}_1(A^{**})$. 
If $b\in Z(pA^{**}p)_+$, then $\|\cv(b)\| = \|b\|$, $\cv(b)p = b$ and $\cv(b)\cv(p) = \cv(b)$.
\end{lemma}

\medskip

In the following lemma, we collect some simple useful facts  concerning Hilbert $C^*$-modules.
Before we give this lemma, let us recall that $E^{**}$ is a Hilbert $A^{**}$-module with the module action and the inner-product extending the ones in $E$.

\medskip

\begin{lemma}
\label{lem:xv}
Let $p\in {\rm Proj}_1(A^{**})$, $\delta\in [0,1)$ and $x\in E\setminus \{0\}$.
Set $a := \frac{\langle x,x
\rangle}{\|x\|^2}$, $q_\delta:=e_a(\delta,1]$, $q_x := e_a(0,1]$ and $F_\Phi := \overline{\Phi(E)}$.

\smnoind (a) If $p$ is open and $y\in E$ satisfying
$\left<x,y\right>p=0$, then
$\left<\Phi(x),\Phi(y)\right>p = 0$.

\smnoind (b) If $v\in A^{**}$ such that $\langle x,x \rangle v
\in A$, then $xv\in E$.

\smnoind (c) If $u,v\in A^{**}$ with $au = av$, then $q_\delta u =q_\delta v$.
Thus, $ap = a$ will imply that $q_x\leq p$.

\smnoind (d) $xp = x$ if and only if $a\in pAp$, which is also equivalent to $x\in
E\cdot (A \cap pA^{**}p).$

\smnoind (e) $xq_x =x$ and $\Phi(x)q_x = \Phi(x)$.

\smnoind (f) $F_\Phi \cdot I_E = F_\Phi$ and $I_{F_\Phi} \subseteq I_E$.
\end{lemma}
\begin{proof}
In the following, we consider $\{e_n\}_{n\in \mathbb{N}}$ to be an approximate unit in $C^*(a)$.
Notice that $\|xe_n - x\| \to 0$ since $\|x - x e_n\|^2 = \|x\|^2\|a - e_n a - a e_n + e_n a e_n\|$. 

\smnoind
(a) Pick any increasing net
$\{a_\lambda\}$ in $A_+ \cap pA^{**}p$ with $a_\lambda\uparrow
p$ (note that $p$ is open).
As $a_\lambda = p a_\lambda$, one has $\left<x,
ya_\lambda\right> = 0$ (for any $\lambda$). Thus,
$\left<\Phi(x),\Phi(y)\right>a_\lambda = 0$ (for any $\lambda$), and hence 
$\left<\Phi(x),\Phi(y)\right>p = 0$.

\smnoind (b) As $e_n v\in A$ (by the hypothesis) and
$\|xv - xe_n v\|^2_{E^{**}}
 = \|x\|^2\|v^*(1- e_n) a (1- e_n) v\|$,
we see that $xv\in E$.

\smnoind (c) Let $\{b_n\}$ be a sequence in $C^*(a)_+$ such that $b_n\uparrow q_\delta$.
As $b_n(u-v) = 0$ ($n\in \mathbb{N}$), we see that
$q_\delta u = q_\delta v$.
By taking $\delta = 0$, we obtain also the second statement.

\smnoind (d) If $xp = x$, then $a = p a p$. 
If $a\in pAp$, then $e_n \in pAp$ and $x\in E\cdot (A \cap pA^{**}p)$ (as $\|x e_n - x\|\to 0$).
Finally, if $x\in E\cdot (A \cap pA^{**}p)$, then clearly $xp = x$.

\smnoind (e) As $xe_n = xe_nq_x \to xq_x$ in norm, one has $x =
xq_x$. 
Now, part (c) implies that $x = zb$ for some $z\in E$ and $b\in A \cap q_xA^{**}q_x$.
Thus,
$\Phi(x) = \Phi(z)b \in F\cdot (A \cap q_xA^{**}q_x)$, which gives $\Phi(x)q_x = \Phi(x)$.

\smnoind (f) As $E$ is a Hilbert $I_E$-module, any $z\in E$ is of the form $z = ya$ for some $y\in E$ and $a\in I_E$.
Thus, $\Phi(E) \subseteq F_\Phi\cdot I_E$.
The second statement follows from the first one (as $I_E$ is an ideal of $A$).
\end{proof}

\medskip

\section{The main results}

\medskip

We may now start proving our main theorem. 
Observe that in the proof for
the real rank zero case in \cite{LNW-orth-pres-rr0}, one starts with
an element $x\in E$ with $p_x:=\langle x, x\rangle$ being a
projection, and shows that one can find $w_x\in Z(p_xAp_x)_+$
such that
$\left<\Phi(y),\Phi(z)\right> = \left<y,z\right>w_x$ ($y,z\in
E\cdot (p_xAp_x)$).
Since there are plenty of such $x$'s when $A$ has real rank zero, we
can ``patched together'' $\cv(w_x)$, where $x$ runs through a
``maximal disjoint'' family of such elements, and then do a surgery to find the
required $u$.

\medskip

However, a general $C^*$-algebra $A$ might not even have any
projection. 
Therefore, our starting point is the following formally weaker lemma (notice that only $y$ is allowed to vary).
After obtaining this lemma, we will then
``patch together'' a different set of elements, and do a surgery to
obtain our main theorem.

\medskip

\begin{lemma}
\label{lem:uepsilon} Suppose that $x\in E\setminus \{0\}$.
If $a:= \frac{\langle x, x \rangle}{\|x\|^2}$ and $q_x := e_a(0,1]$, there is $u_x\in Z(q_xA^{**}q_x)_+$ such that
\begin{equation*}
\left<\Phi(y),\Phi(x)\right>
\ =\ \left<y,x\right>u_x
\quad (y\in E).
\end{equation*}
\end{lemma}
\begin{proof}
Without loss of generality we assume that $\|x\| = 1$.
If $\epsilon\in (0,1)$ and $q_\epsilon := e_a(\epsilon, 1]$, pick any $b\in C^*(a)_+$ satisfying ${q_\epsilon}\leq ab\leq 1$ and set $x_\epsilon := xb^{1/2}\in E$. 
Then we have $\left<x_\epsilon
{q_\epsilon},x_\epsilon \right>_{A^{**}} =\left<x_\epsilon, x_\epsilon{q_\epsilon}\right>_{A^{**}} = \left<x_\epsilon  {q_\epsilon},x_\epsilon
{q_\epsilon}\right>_{A^{**}} = {q_\epsilon}$. 
Moreover, 
\begin{equation}
\label{eqt-lim}
b^{1/2} q_\epsilon (b^{1/2} + q_\epsilon/n)^{-1} \uparrow q_\epsilon \quad \textrm{when}\quad n \to \infty.
\end{equation}

Put $u_{\epsilon} := \left<\Phi(x_\epsilon),\Phi(x_\epsilon) \right>
{q_\epsilon} \in A{q_\epsilon}$. 
Consider $c\in q_\epsilon A^{**}q_\epsilon\cap A_+$ to be a norm one element, and set $p :=
e_c(\alpha,\beta)\in {q_\epsilon}A^{**}{q_\epsilon}$ for some
$\alpha < \beta$ in $\mathbb{R}_+$. Let $b_n \in C^*(c)\subseteq
A \cap q_\epsilon A^{**}q_\epsilon$ such that $0\leq  b_n \uparrow p$ and $b_n b_{n+1}
= b_n$ ($n\in \mathbb{N}$).
Set $c_n := 1-b_n$, and observe that
$1\geq c_n \downarrow(1-p)$, $b_n c_{n+k} =0$, $b_np=b_n$,  and
$c_{n+k}(1-p)=1-p$ ($n, k\in \mathbb{N}$). 
Since
$$
\left<x_\epsilon b_n,x_\epsilon c_{n+k}\right> \ =\
b_n{q_\epsilon}\left<x_\epsilon,x_\epsilon\right> c_{n+k} \ =\ b_n
{q_\epsilon} c_{n+k}\ =\ b_n c_{n+k}\ =\ 0,
$$
we have $b_n u_{\epsilon} c_{n+k} = \left<\Phi(x_\epsilon
b_n),\Phi(x_\epsilon c_{n+k})\right>q_\epsilon =0$ (by Lemma \ref{lem:xv}(a)).
By letting $k \to \infty$
and then $n \to \infty$, we see that $pu_{\epsilon} (1-p) =0$, i.e.,
$pu_{\epsilon}= pu_{\epsilon} p$. 
Similarly, we have $p u_\epsilon p
= u_{\epsilon} p$ and so, $p u_\epsilon = u_\epsilon p$. As $c$ can
be approximated in norm by linear combinations of projections of the
form $e_c(\alpha,\beta)$, one concludes that $u_{\epsilon}$ commutes
with an arbitrary element in $A \cap q_\epsilon A^{**}q_\epsilon$.
Thus,
$u_{\epsilon}$ commutes with elements in
${q_\epsilon}A^{**}{q_\epsilon}$ (as $q_\epsilon$ is open). In particular, $u_{\epsilon} =
u_{\epsilon}q_{\epsilon} = {q_\epsilon}u_{\epsilon}{q_\epsilon}=
{q_\epsilon}\left<\Phi(x_\epsilon),\Phi(x_\epsilon)
\right>{q_\epsilon} \in {q_\epsilon}A{q_\epsilon}$, which means
that $u_{\epsilon}\in Z({q_\epsilon}A^{**}{q_\epsilon})_+$.

For any $y\in E$, the element $y - x_\epsilon \left<x_\epsilon ,y\right> \in E$ is
orthogonal to $x_\epsilon  {q_\epsilon}\in E^{**}$. By Lemma \ref{lem:xv}(a),
we have
$$
\left<\Phi(y),\Phi(x_\epsilon )\right>{q_\epsilon}\ =\
\left<y,x_\epsilon \right>\left<\Phi(x_\epsilon ), \Phi(x_\epsilon
)\right>{q_\epsilon}
\ =\ \left<y,x_\epsilon
\right>u_{\epsilon},
$$
which implies that $\left<\Phi(y),\Phi(x)\right>b^{1/2}{q_\epsilon}
= \left<y,x\right>u_{\epsilon}b^{1/2}{q_\epsilon}$ (because
$b^{1/2}{q_\epsilon} = {q_\epsilon}b^{1/2}{q_\epsilon}\in q_\epsilon A^{**}q_\epsilon$).
Now Relation
\eqref{eqt-lim} tells us that 
\begin{align}\label{eq:q-epsilon}
\left<\Phi(y),\Phi(x)\right>{q_\epsilon} \ =\
\left<y,x \right>{u_\epsilon} \qquad (y\in E).
\end{align}

If  $0< \delta \leq \epsilon<1$, we have
$q_\epsilon\leq q_\delta$ and $q_\epsilon A^{**}q_\epsilon\subseteq
q_\delta A^{**}q_\delta$.
Hence, 
$$au_{{\delta}}{q_\epsilon} 
\ =\ \left<x,x\right>u_{{\delta}}{q_\epsilon}
\ =\ \left<\Phi(x),\Phi(x)\right>q_\delta{q_\epsilon}
\ =\ \left<\Phi(x),\Phi(x)\right>{q_\epsilon}
\ =\ a u_\epsilon,$$ 
and Lemma \ref{lem:xv}(c) tells us that
$
u_\delta q_\epsilon
 = q_\delta u_\delta q_\epsilon
 = q_\delta u_\epsilon
 = q_\delta q_\epsilon u_\epsilon
 = u_\epsilon.
$
By taking adjoint, we see that $u_\delta$ commutes with $q_\epsilon$, which gives
\begin{equation}\label{rel:u-dec}
0\ \leq\ u_\epsilon\ =\
u_\delta^{1/2}q_\epsilon u_\delta^{1/2}\ \leq\ u_\delta \qquad (0<\delta
\leq \epsilon< 1).
\end{equation}

Next, we show that $\{u_\epsilon\}_{\epsilon\in (0,1)}$ is a bounded
set. Suppose on the contrary that there is a decreasing sequence
$\{\epsilon_n\}_{n\in \mathbb{N}}$ with $\|u_{\epsilon_n}\| >
\|u_{\epsilon_{n-1}}\| + n^5$ for every $n\in \mathbb{N}$ (see Relation \eqref{rel:u-dec}). Let $b_n, d_n\in C^*(a)_+$ such that
$e_a(\epsilon_{4n-1}, \epsilon_{4n-2}]\leq b_n \leq
e_a(\epsilon_{4n}, \epsilon_{4n-3}]$ ($\leq q_{\epsilon_{4n}}$) and
$q_{\epsilon_{4n}}\leq ad_n \leq 1$. As $b_n, q_{\epsilon_{4n-1}},
q_{\epsilon_{4n-2}}\in q_{\epsilon_{4n}}A^{**}q_{\epsilon_{4n}}$ and
$u_{\epsilon_{4n}}\in
Z(q_{\epsilon_{4n}}A^{**}q_{\epsilon_{4n}})_+$, we see that
$$\|u_{\epsilon_{4n}}b_n\|\ \geq\
\|u_{\epsilon_{4n}}(q_{\epsilon_{4n-1}} - q_{\epsilon_{4n-2}})\| \ =
\ \|u_{\epsilon_{4n-1}} - u_{\epsilon_{4n-2}}\|\ \geq \ (4n-1)^5.$$
If $x_n:= xb_n^{1/2}d_n^{1/2}$, then $\langle x_n, x_n\rangle =
b_nq_{\epsilon_{4n}} ad_n = b_n$. Moreover, if
$m\neq n$, then
$$\langle x_n,x_m\rangle
\ =\ d_n^{1/2}b_n^{1/2}e_a(\epsilon_{4n}, \epsilon_{4n-3}] a e_a(\epsilon_{4m}, \epsilon_{4m-3}]b_m^{1/2}d_m^{1/2}
\ =\ 0$$
(as $(\epsilon_{4n}, \epsilon_{4n-3}]\cap (\epsilon_{4m},
\epsilon_{4m-3}] = \emptyset$). Let $y:= \sum_{n=1}^\infty
x_n/n^2\in E$ (note that $\|x_n\|^2 = \|b_n\| \leq 1$). 
For any $m\in \mathbb{N}$, we have 
$\langle \Phi(y), \Phi(y)\rangle
\geq \langle\Phi(x_m),\Phi(x_m)\rangle/m^4$ (as $\Phi$ preserves
orthogonality), and by Relation \eqref{eq:q-epsilon},
\begin{eqnarray}\label{eqt:orth-pre-norm}
m^4 \langle \Phi(y), \Phi(y)\rangle
& \geq & \langle\Phi(x_m),\Phi(x_m)\rangle
\quad = \quad \langle \Phi(x_m),\Phi(x)\rangle q_{\epsilon_{4m}} b_m^{1/2}d_m^{1/2}\\
& = & \langle x_m,x\rangle u_{\epsilon_{4m}} b_m^{1/2}d_m^{1/2}
\quad = \quad b_mu_{\epsilon_{4m}}\nonumber
\end{eqnarray}
(since $b_m^{1/2}d_m^{1/2}\in q_{\epsilon_{4m}} A^{**}q_{\epsilon_{4m}}$ and
$u_{\epsilon_{4m}}\in
Z(q_{\epsilon_{4m}}A^{**}q_{\epsilon_{4m}})_+$). Consequently,
$\|\Phi(y)\|^2 \geq (4m-1)^5/m^4$ for all $m\in \mathbb{N}$, which is a
contradiction.

Now, the bounded sequence $\{u_{1/n}\}_{n\in \mathbb{N}}$ in
$(q_xA^{**}q_x)_+$ has a subnet having a weak-*-limit $u_x\in
(q_xA^{**}q_x)_+$. 
As $q_{1/n} \uparrow q_x$, we have
$\bigcup_{n\in \mathbb{N}} q_{1/n} A^{**} q_{1/n}$ being
weak-*-dense in $\bigcup_{n\in \mathbb{N}} q_{1/n} A^{**} q_x$ and
hence also weak-*-dense in $q_xA^{**}q_x$. Thus, $u_x\in Z(q_x
A^{**} q_x)_+$ (as $q_{1/m} u_x = u_x q_{1/m} = u_{1/m}\in Z(q_{1/m} A^{**}
q_{1/m})$ for any $m\in \mathbb{N}$).
By Relation \eqref{eq:q-epsilon} and
Lemma \ref{lem:xv}(e), we have $\left<\Phi(y),\Phi(x)\right> =
\left<\Phi(y),\Phi(x)\right>q_x = \left<y,x\right>u_x$ ($y\in E$).
\end{proof}

\medskip

\begin{theorem}\label{thm-unbdd}
Suppose that $\Phi:E\to F$ is a $\mathbb{C}$-linear map (not assumed to be bounded).
Then $\Phi : E\to F$ is an orthogonality
preserving $A$-module map if and only if
there exists $u\in Z(M(I_E))_+$ (where $I_E\subseteq A$ is the ideal generated by the inner products of elements in $E$) such that
$$
\left<\Phi(x),\Phi(y)\right>\ =\ u\left<x,y\right>\quad (x,y\in E).
$$
In this case, $u$ is unique and $\Phi$ is automatically bounded.
\end{theorem}
\begin{proof}
As $E$ is a full Hilbert $I_E$-module, it is easy to see that $u$ is unique if it exists, and in this case, $\|\Phi\|^2 \leq \|u\|$.

The sufficiency is obvious, and we will establish the necessity in the following. 
Since
$I_{F_\Phi}\subseteq I_E$  (see Lemma \ref{lem:xv}(f)), by replacing
$\Phi$ with the induced map $\Phi_0: E \to F_\Phi:= \overline {\Phi(E)}$, we may assume
that $I_E = A$.

Let $M$ be a maximal family of orthogonal norm-one elements in $E$,
and $\mathcal{F}$ be the collection of all non-empty finite subsets
of $M$. 
If $\{y,z\}\in \mathcal{F}$, then by
Lemma \ref{lem:uepsilon},
\begin{align*}
\langle y, y\rangle u_y\ =\ \langle \Phi(y), \Phi(y)\rangle
\ =\ \langle \Phi(y), \Phi(y + z)\rangle
\ =\ \langle y, y \rangle u_{y+z},
\end{align*}
which implies that $\|y(u_{y+z} - u_y)\|_{E^{**}}^2 \leq \|u_{y+z} - u_y\| \|\langle
y,y\rangle (u_{y+z} - u_y)\| = 0$, and so,
\begin{equation}\label{eqt:yuy+z}
yu_y\ =\ yu_{y+z}.
\end{equation}
Moreover, $\langle y, y\rangle q_{y+z} = \langle
y, y+z \rangle q_{y+z} = \langle y, y\rangle$ (by Lemma
\ref{lem:xv}(e)) and thus $q_y\leq q_{y+z}$ (by Lemma
\ref{lem:xv}(c)). 
On the other hand, if $p\in {\rm Proj}_1(A^{**})$
such that $q_y\leq p$ and $q_z \leq p$, then $\left< y+z, y+z
\right> p = \left< y , y \right> q_yp + \left< z, z \right> q_zp =
\left< y+z, y+z \right>$, which tells us that $q_{y+z}\leq p$ (again by Lemma \ref{lem:xv}(c)).
Thus, $q_{y+z} = q_y \vee q_z$ in
${\rm Proj}_1(A^{**})$. 
Inductively, if $S\in \mathcal{F}$ and $x_S := \sum_{x\in S} x$, then by Lemma \ref{lem:uepsilon} as Relation \eqref{eqt:yuy+z}, 
\begin{equation}
\label{eqt-v-os}
\left<\Phi(y),\Phi(x)\right> \ =\ \left<y,x\right>u_{x}
\ =\ \left<y,x\right> u_{x_S} \quad (y\in
E; x\in S), 
\end{equation}
\begin{equation}\label{eqt:proj-sup}
q_{x_S}\ =\ {\bigvee}_{x\in S} q_{x} 
\quad (\text{as elements in } {\rm Proj}_1(A^{**}). 
\end{equation}

If $S'\in \mathcal{F}$ with $S\subseteq S'$, then $\langle
x_S,x_S\rangle u_{x_{S'}} = \langle \Phi(x_S), \Phi(x_S)\rangle =
\langle x_S, x_S\rangle u_{x_S}$ (by Relation \eqref{eqt-v-os}).
Thus, Lemma \ref{lem:xv}(c) tells us that
\begin{equation}\label{eqt:q-u}
q_{x_S} u_{x_{S'}}\ =\ q_{x_S} u_{x_S}\ =\ u_{x_S}.
\end{equation}
By taking adjoint, we see that $q_{x_S}$ commutes with $u_{x_{S'}}$, and
Relation \eqref{eqt:q-u} implies that $\{u_{x_S}\}_{S\in \mathcal{F}}$ is an
increasing net in $A^{**}_+$.

We now show that $\{u_{x_S}\}_{S\in \mathcal{F}}$ is a bounded net.
Suppose on the contrary that there is an increasing sequence
$\emptyset \subsetneq S(0) \subsetneq S(1) \subsetneq ...$ in $\mathcal{F}$ with
$$
\|u_{x_{S(n)}}\|\ \geq\ \|u_{x_{S(n-1)}}\| + n^5 \qquad (n\in \mathbb{N})
$$
(notice that $\|u_{x_S}\| \leq \|u_{x_{S'}}\|$ if $S\subseteq S'$).
Denote by 
$y_n:= \sum_{x\in S(n)\setminus S(n-1)} x = x_{S(n)} - x_{S(n-1)}$ ($n\in \mathbb{N}$).
By
\cite[V.1.6]{Takesaki}, one has a partial isometry $w\in A^{**}$
such that $q_{x_{S(n)}} - q_{x_{S(n-1)}} = q_{x_{S(n-1)}}\vee q_{y_n} -
q_{x_{S(n-1)}} = w(q_{y_{n}} - q_{x_{S(n-1)}}\wedge q_{y_n})w^*$, 
which implies 
\begin{eqnarray*}
u_{x_{S(n)}}  
\quad = \quad u_{x_{S(n)}}^{1/2}q_{x_{S(n)}}u_{x_{S(n)}}^{1/2}
& \leq & u_{x_{S(n)}}^{1/2}(q_{x_{S(n-1)}} + w q_{y_n} w^*)u_{x_{S(n)}}^{1/2}\\
& = & u_{x_{S(n-1)}} + u_{x_{S(n)}}^{1/2}wq_{y_n}w^*u_{x_{S(n)}}^{1/2}
\end{eqnarray*}
(see also \eqref{eqt:q-u}). 
On the other hand, by \eqref{eqt:proj-sup} and Lemma \ref{lem:cp}(b),
\begin{eqnarray*}
\lefteqn{u_{x_{S(n)}}^{1/2}wq_{y_n}w^*u_{x_{S(n)}}^{1/2}}\qquad \\
& = & \cv(u_{x_{S(n)}}^{1/2})q_{x_{S(n)}}wq_{y_n}w^*q_{x_{S(n)}}
\cv(u_{x_{S(n)}}^{1/2})
\ =\ q_{x_{S(n)}}wq_{y_n}\cv(u_{x_{S(n)}}^{1/2})
\cv(u_{x_{S(n)}}^{1/2})w^*q_{x_{S(n)}}\\
& = & q_{x_{S(n)}}wq_{y_n}q_{x_{S(n)}}
\cv(u_{x_{S(n)}}^{1/2})\cv(u_{x_{S(n)}}^{1/2})w^*q_{x_{S(n)}}
\ =\ q_{x_{S(n)}}wq_{y_n}u_{x_{S(n)}}w^*q_{x_{S(n)}}.
\end{eqnarray*}
Consequently, $u_{x_{S(n)}} - u_{x_{S(n-1)}} \leq q_{x_{S(n)}}wq_{y_n}u_{x_{S(n)}}w^*q_{x_{S(n)}}$,
which gives $\|q_{y_n}u_{x_{S(n)}}\| > n^5$. Let $a_n :=
\frac{\langle y_n,y_n\rangle}{\|y_n\|^2}$.
Since $\{a_nb: b\in C^*(a_n)\}$ is a norm-dense ideal of $C^*(a_n)$, there
is $b_n\in C^*(a_n)_+$ such that 
$$\|a_nb_n\|\ \leq\ 1 \qquad \text{and} \qquad \|a_n
b_n u_{x_{S(n)}}\|\ >\ n^5.$$
Define $x_n:= y_nb_n^{1/2}/\|y_n\|$. Then
clearly $\{x_n\}_{n\in\mathbb{N}}$ is an orthogonal sequence with
$\langle x_n, x_n\rangle = a_nb_n$. Let $z:= \sum_{n=1}^\infty
x_n/n^2\in E$ (notice that $\|x_n\| \leq 1$).
As in \eqref{eqt:orth-pre-norm}, since $\Phi$ preserves
orthogonality, for any $m\in \mathbb{N}$,
$$
\langle \Phi(z), \Phi(z)\rangle
\ \geq \ b_m^{1/2}\langle y_m,y_m\rangle u_{x_{S(m)}} b_m^{1/2}/(m^4 \|y_m\|^2)
\ = \ a_m b_m u_{x_{S(m)}}/m^4
$$
(because of Relation \eqref{eqt-v-os} as well as the facts that $b_m^{1/2}\in q_{x_{S(m)}}A^{**}q_{x_{S(m)}}$ and $u_{x_{S(m)}}\in Z(q_{x_{S(n)}} A^{**} q_{x_{S(n)}})_+$).
This gives the
contradiction that $\|\Phi(z)\|^2 > m$ for all $m\in \mathbb{N}$.

For any $x\in E$, we set $v_x:= \cv(u_x)$. By Lemmas
\ref{lem:uepsilon}, 
\ref{lem:cp}(b) and \ref{lem:xv}(e), we have
\begin{align}
\label{eqt:theta-ux}
\left<\Phi(y),\Phi(x)\right>\ =\
\left<y,x\right>q_xv_x \ = \ \left<y,x\right> v_x \qquad (y\in E).
\end{align}
Moreover, by Lemma \ref{lem:cp}(b), the net $\{v_{x_S}\}_{S\in \mathcal{F}}$ is also bounded.
Let $v\in Z(A^{**})_+$ be the weak-*-limit of a subnet of
$\{v_{x_S}\}_{S\in \mathcal{F}}$.
Note that if $S\in \mathcal{F}$ and $x\in S$, then by Lemmas \ref{lem:xv}(e) and \ref{lem:cp}(b) as well as Relations \eqref{eqt:proj-sup} and \eqref{eqt:q-u}, we have $\langle y, x \rangle v_{x_S} = \langle y, x \rangle q_x q_{x_S}v_{x_S} = \langle y, x \rangle u_x = \left<\Phi(y),\Phi(x)\right>$ ($y\in E$).
Therefore, 
\begin{align}\label{eq:ideal-u} \left<\Phi(y),\Phi(x)\right>\ = \
\left<y,x\right>v \quad (y\in E, x\in M).
\end{align}
If $I$ is the ideal of $A$ generated by $\{\left<y,x\right>: y\in E, x\in
M\}$, then $Iv\subseteq A$.
For any $z\in E\cdot I\setminus \{0\}$, one has $zv\in E$. 
On the other hand, as $\langle z, z \rangle v_z\in A$ (see
\eqref{eqt:theta-ux}), we know that $zv_z\in E$ (by Lemma \ref{lem:xv}(b)).
Furthermore, one has $\left<x,z\right> v_z =
\left<\Phi(x),\Phi(z)\right>
 = v \left<x,z\right> = \left<x,z\right> v$ if $x\in M$.
This shows that the element $z(v-v_z)$ in $E$ is orthogonal to any $x\in M$. 
This forces $zv = zv_z$ (by the maximality of $M$).
As a consequence,
$$
\left<\Phi(x),\Phi(y)\right>a\ =\ \left<x,ya\right>v_{ya} \ = \ \left<x,y\right>av \quad(x,y\in E, a\in I).
$$
If $q$ is the central open projection in $A^{**}$ with $I =
A\cap qA^{**}q$, then $q$ is the weak-*-limit of a net in $I$, and we have
\begin{align}\label{eq:zuz}
\left<\Phi(x),\Phi(y)\right>q\ =\ v \left<x,y\right>q \quad(x,y\in E).
\end{align}

We now claim that $\phi: a\mapsto qa$ is an injection from $A$ onto
$qA$. Indeed, if $a\in \ker\phi$, then $\langle x,
ya \rangle = \langle x, y \rangle qa = 0$ (for every $x\in M$ and $y\in E$), and
the maximality of $M$ as well as the fullness of $E$ will imply that
$a=0$. Consequently, $\phi$ induces a $*$-isomorphism $\tilde \phi:
M(A) \rightarrow M(qA)$. By Equation \eqref{eq:zuz} and the fullness
of $E$, we see that $v$ induces an element $m\in Z(M(qA))_+$ such
that $q\left<\Phi(x),\Phi(y)\right> = m (q\left<x,y\right>)$ $(x,y\in
E)$. If $u := (\tilde \phi)^{-1}(m)$, then $u\in Z(M(A))_+$ and the
injectivity of $\phi$ gives the required relation
$$
\left<\Phi(x),\Phi(y)\right>\ =\ u \left<x,y\right> \quad(x,y\in E).
$$\end{proof}

\medskip

\begin{remark}\label{rem:rt-mult}
(a) We denote by $u_\Phi$ the unique element in
$Z(M(I_E))_+$ associated with $\Phi$ as in Theorem \ref{thm-unbdd},
and we set $w_\Phi := u_\Phi^{1/2}$. 

\smnoind
(b) Suppose that $v\in M(I_E)$. 
Since $E$ is a Hilbert $I_E$-module, it becomes a unital right Banach $M(I_E)$-module in a canonical way. 
We denote by $R_v: E\to E$ the right multiplication of $v$, i.e.\ $R_v(x) = xv$ ($x\in E$). 
\end{remark}

\medskip

\begin{corollary}\label{cor:adj-inj}
Suppose that $\Phi$ is an orthogonality preserving $A$-module map.

\smnoind
(a) $I_{F_\Phi} = \overline{u_\Phi I_E}$ and
$\ker \Phi = \ker R_{w_\Phi}$.
Moreover, there is a Hilbert
$A$-module isomorphism $\Theta: \overline{Ew_\Phi} \to F_\Phi$ such that $\Phi = \Theta\circ R_{w_\Phi}$.
Consequently, the induced map $\Phi_0: E\rightarrow F_\Phi$ is adjointable with $\Phi_0^*$ being orthogonality preserving.

\smnoind
(b) If $\Phi$ is injective, then $\Phi^{-1}: \Phi(E) \to E$ is also orthogonality preserving.

\smnoind
(c) If $I_{F_\Phi} = I_E$, then $E{w_\Phi}$ is dense in $E$ and $\Phi$ is injective.
\end{corollary}
\begin{proof}
(a) The first equality follows directly from Theorem \ref{thm-unbdd}.
As
$\|\Phi(x)\| = \|R_{w_\Phi}(x)\|$ ($x\in E$), we see that $\ker \Phi
= \ker R_{w_\Phi}$.
Thus, we can define $\Theta: Ew_\Phi \to F$ by
$\Theta(R_{w_\Phi}(x)) := \Phi(x)$.
Since $\Theta$ preserves the
$A$-valued inner products, it extends to a Hilbert $A$-module isomorphism from $\overline{Ew_\Phi}$ onto $F_\Phi$ that satisfies the required condition.
Furthermore, it is easy to see that both $R_{w_\Phi}: E\to \overline{Ew_\Phi}$ and $\Theta$ are adjointable, and so is $\Phi_0$.
Finally, as $\Phi_0^* = R_{w_\Phi}\circ \Theta^{-1}$, we see that $\Phi_0^*$ also preserves orthogonality.

\smnoind (b) Suppose that $a\in I_E$ with $au_\Phi = 0$.
Then $aw_\Phi = 0$ as $w_\Phi\in C^*(u_\Phi)$ and so, $xa \in \ker \Phi$ for any $x\in E$ (by part (a)).
As $\Phi$ is injective and $E$ is a full Hilbert $I_E$-module, we have $a = 0$.
Consequently, if $x,y\in E$ satisfying $\langle \Phi(x), \Phi(y) \rangle = 0$, then by Theorem \ref{thm-unbdd}, $\langle x, y \rangle = 0$.

\smnoind (c) Part (a) tells us that $u_\Phi I_E$ is dense in $I_{F_\Phi} = I_E$,
and so, $w_\Phi I_E \supseteq w_\Phi (w_\Phi I_E)$ is dense
in $I_E$. Consequently, $Ew_\Phi = (E\cdot I_E)w_\Phi $ is dense in $E$.
By part (a) again, we see that $E$ is isomorphic to $F_\Phi$. Moreover,
if $x\in \ker  R_{w_\Phi}$, then $\langle x, y w_\Phi\rangle =
\langle xw_\Phi, y\rangle = 0$ for any $y\in E$, which implies that $x =
0$.
Consequently, part (a) tells us that $\ker \Phi = \{0\}$.
\end{proof}

\medskip

By Corollary \ref{cor:adj-inj}(a), if $\Phi:E\rightarrow F$ is an
orthogonality preserving $A$-module map with dense range, then $F$
and $\Phi$ can be represented by an element $w_\Phi\in Z(M(I_E))_+$,
up to an isomorphism.
On the other hand, $\Phi$ may not have closed range even if it is injective (see Example \ref{eg-I}(b) below), and Corollary \ref{cor:adj-inj}(b) does not give us any good information about $\Phi^{-1}$. 
Furthermore, it is not true that all orthogonality preserving $A$-module maps are adjointable (see Example \ref{eg-I}(c) below), and it is only true if we restrict the range of the map.

\medskip

\begin{theorem}\label{thm:I+mod+op=isom}
Let $\Phi: E\rightarrow F$ be an orthogonality preserving $A$-module map (not assumed to be bounded), $F_\Phi := \overline{\Phi(E)}$ and $I_E$ be the ideal generated by the inner products of elements in $E$. 

\smnoind
(a) If $I_{F_\Phi} = I_E$, there is a Hilbert $A$-module isomorphism $\Theta: E\to F_\Phi$ such that $\Phi(x) = \Theta(xw_\Phi)$ ($x\in E$).

\smnoind
(b) If $\Phi$ is bijective, then $I_F = I_E$ and there is a unique
invertible $w\in Z(M(I_E))_+$ such that $x\mapsto \Phi(x)w^{-1}$ is a Hilbert $A$-module isomorphism from $E$ onto $F$.
\end{theorem}
\begin{proof}
\smnoind
(a) This follows directly from Corollary \ref{cor:adj-inj}.

\smnoind (b) By Lemma \ref{lem:xv}(f), we have $I_F\subseteq I_E$
and we might assume that $E$ is full.
Notice that $\Phi^{-1}: F\to E$ is an orthogonality preserving
$A$-module map because of
Corollary \ref{cor:adj-inj}(b).
Thus, Theorem \ref{thm-unbdd} gives
$u_{\Phi^{-1}}\in Z(M(I_F))_+$ such that
$$\langle x, y \rangle
\ = \ \langle \Phi^{-1}(\Phi(x)), \Phi^{-1}(\Phi(y)) \rangle \ = \
u_{\Phi^{-1}}u_{\Phi} \langle x, y \rangle \qquad (x,y\in E).$$ As
$E$ is full, the above implies that for any $a\in A$, one has $a =
u_{\Phi^{-1}}u_{\Phi}a \in u_{\Phi^{-1}}I_F \subseteq I_F$ (by Corollary \ref{cor:adj-inj}(a)).
This
shows that $I_F = A$ and $u_\Phi$ is invertible (and so is
$w_\Phi$). Now, part (b) follows directly from part (a) (note that
the uniqueness of $w$ follows from the uniqueness of $u_\Phi$).
\end{proof}

\medskip

We remark that in the case of complex Hilbert spaces (i.e.,  $A= \mathbb{C}$), the condition that $I_{\overline{\Phi(E)}} = I_E$ is the same as $\Phi$ being nonzero.
However, in the general case, one cannot even replace the requirement $I_{{\overline{\Phi(E)}}} = I_E$ in Theorem \ref{thm:I+mod+op=isom}(a) to
$\Phi$ being either injective or surjective (see Example \ref{eg-I}(a)\&(d) below; note that a Hilbert $A$-module isomorphism is isometric).
We remark also that even in the situation of Theorem \ref{thm:I+mod+op=isom}(a), the submodule $\Phi(E)$ need not be closed in $F$ and $w_\Phi$ need not be invertible (see Example \ref{eg-I}(b) below).

\medskip

\begin{example}
\label{eg-I}
(a) Let $A := C[0,1]$, $E := C[0,1]$ and $F := C_0(0,1]$. 
If $a\in A_+$ is given by $a(t) := t$ ($t\in [0,1]$) and $\Phi: E\to F$ is defined by $\Phi(x) := xa$, then $\Phi$ is an injective orthogonality preserving $A$-module map. 
However, there is no isometric $A$-module map from $E$ into $F$. 
Suppose on the contrary that $\Theta: E\to F$ is such a map. 
Then $\Theta(b) = \Theta(1)b$ ($b\in A$). 
Since $f:=\Theta(1)$ is in $C_0(0,1]$, one can find $t_0\in (0,1)$ such that $|f(t)| < 1/2$ for $t\leq t_0$. 
Now, if $b\in A$ such that $\|b\| =1$ and $b$ vanishes on $[t_0,1]$, then $\|\Theta(b)\| \leq 1/2 < 1 =\|b\|$ which is a contradiction. 

\smnoind (b) Let $A := C_0(0,1]$ and $a\in A_+$ be the function
defined by $a(t) := t$ ($t\in (0,1]$). If we set $E := A$ and $F:=A$,
and define $\Phi: E \to F$ by $\Phi(x) := xa$, then $\Phi$ is an
orthogonality preserving $A$-module map with dense range and
$I_{F_\Phi} = A = I_E$, but $\Phi$ is not surjective, and $a =
w_\Phi$ is not invertible in $M(A)$.

\smnoind
(c) Let $A := C_0(0,1)$, $E:=\{f\in A: f(1/2) =0\}$, $F:=A$ and $\Phi: E\to F$ be the canonical injection.
Then $\Phi$ is an orthogonality preserving $A$-module map with closed range and $I_{F_\Phi} = I_E$, but $\Phi$ is not an adjointable map from $E$ into $F$.
Indeed, suppose that $\Phi$ is adjointable, and $g\in F$ with $g(1/2)\neq 0$.
Then $\langle \Phi^*(g) , f \rangle_E - \langle g , f \rangle_F = 0$ for any $f\in E\subseteq F$, which implies that $\Phi^*(g) - g = 0$ (because $0$ is the only element in $F$ being orthogonal to $E$).
Thus, we have a contradiction $g = \Phi^*(g)\in E$.

\smnoind
(d) Let $A = \mathbb{C}\oplus \mathbb{C}$, $E = A$ and $F = \mathbb{C}\oplus 0\subseteq E$. 
Define $\Phi(x) := x(1,0)$ (for any $x\in E$). 
Then $\Phi$ is a surjective orthogonality preserving $A$-module map, but $E \ncong F$. 
\end{example}

\medskip

\section{Extending orthogonality preservers to the linking algebras}

\medskip

For any $x,y\in E$, we define an operator $\theta_{y,x}$ by
$\theta_{y,x}(z):= y \langle x, z\rangle$ ($z\in E$). As usual, we
denote by $\CK(E)$ the closed ideal generated by $\{\theta_{y,x}:
x,y\in E\}$ in the $C^*$-algebra of adjointable maps from $E$
into itself.

\medskip

Let $\tilde E$ be the conjugate Banach space of $E$. 
Recall, from e.g.\ \cite[1.1]{BGR}, that the
$*$-algebra structure on the linking $C^*$-algebra $\mathfrak{L}_E:=
\Big(\begin{array}{cc}
\CK(E) & E\\
\tilde E & I_E
\end{array}\Big)$
is given by $\Big(\begin{array}{cc}
\theta & x\\
\tilde y & a
\end{array}\Big)^* = \Big(\begin{array}{cc}
\theta^* & y\\
\tilde x & a^*
\end{array}\Big)$ 
and
$\Big(\begin{array}{cc}
\theta & x\\
\tilde y & a
\end{array}\Big)\Big(\begin{array}{cc}
\theta' & x'\\
\tilde y' & a'
\end{array}\Big) 
= \Big(\begin{array}{cc}
\theta\theta' +\theta_{x,y'} & \theta(x') + xa'\\
\widetilde{\theta'(y)+y'a} & \langle y, x'\rangle + aa'
\end{array}\Big).$
We set $J_E:E\to \mathfrak{L}_E$ to be the canonical embedding, i.e.\ $J_E(x) := \Big(\begin{array}{cc}
0 & x\\
0 & 0
\end{array}\Big)$.

\medskip

If $u,v\in Z(M(A))$, there are two 
linear maps $L_{v,u}, R_{v,u}: \mathfrak{L}_E \to \mathfrak{L}_E$
given by $L_{v,u}\Big(\begin{array}{cc}
\theta & x\\
\tilde y & a
\end{array}\Big) = \Big(\begin{array}{cc}
\theta \circ R_v & xv\\
\widetilde{yu} & au
\end{array}\Big)$ and $R_{v,u}\Big(\begin{array}{cc}
\theta & x\\
\tilde y & a
\end{array}\Big) = \Big(\begin{array}{cc}
\theta \circ R_v & xu\\
\widetilde{yv} & au
\end{array}\Big).$
It is easy to check that $M_{v,u}:= (L_{v,u}, R_{v,u})$ is in the  multiplier algebra $M(\mathfrak{L}_E)$.

\medskip

As noted in the introduction, it is not obvious to us how to prove
Theorem \ref{thm-unbdd} by extending an orthogonality preserving
$A$-module map to a disjointness preserving map on the linking
algebras, because it is not clear how one can induce a map from
$\CK(E)$ to $\CK(F)$ that is compatible with $\Phi$.
Nevertheless, after proving Theorem \ref{thm-unbdd} and Corollary \ref{cor:adj-inj}, one can show in
Theorem \ref{thm:link-alg} below that this map can be obtained when $F$ is replaced by $F_\Phi$.
Notice that Theorem \ref{thm:link-alg} is an extension of Theorem
\ref{thm-unbdd} because for any $x,y\in E$, one has, by Relations \eqref{rel:Del-J} and \eqref{rel:Del-zero-prod} below,
\begin{equation}\label{eqt:cp-Phi-Delta}
\Big(\begin{array}{cc}
0 & 0\\
0 & \langle \Phi(x), \Phi(y)\rangle
\end{array}\Big)
\ = \ \Delta(J_E(x))^*\Delta(J_E(y))\\
\ = \ \Big(\begin{array}{cc}
0 & 0\\
0 & u\langle x, y\rangle
\end{array}\Big).
\end{equation}
However, we do not know how
to obtain this result without Theorem \ref{thm-unbdd}.

\medskip

\begin{theorem}\label{thm:link-alg}
Suppose that $\Phi:E\to F$ is an $A$-module map (not assumed to be bounded), and $F_\Phi := \overline{\Phi(E)}$.  
Then $\Phi$ is orthogonality preserving if and only if there exists a linear map $\Gamma: \mathfrak{L}_E \rightarrow
\mathfrak{L}_{F_\Phi}$ (respectively, $\Delta: \mathfrak{L}_E \rightarrow \mathfrak{L}_{F_\Phi}$) such that 
\begin{equation}\label{rel:Del-J}
\Gamma \circ J_E\ =\ J_{F_\Phi}\circ \Phi
\quad (\text{respectively, } \Delta \circ J_E \ = \ J_{F_\Phi}\circ \Phi),
\end{equation}
and for any $c,d\in \mathfrak{L}_E$ 
\begin{equation}\label{eqt:disj-pres}
cd = 0 \ \Rightarrow\ \Gamma (c)\Gamma(d)\ =\ 0
\quad (\text{respectively, } c^*d=0 \ \Rightarrow\ \Delta (c)^*\Delta(d) = 0). 
\end{equation}
In this case, one can find $\Gamma$ and $\Delta$ satisfying \eqref{rel:Del-J} as well as $u\in Z(M(I_E))_+$ such that 
for any $c,d\in \mathfrak{L}_E$, 
\begin{equation}\label{rel:Del-zero-prod}
\Gamma (c)\Gamma(d)\ =\ M_{u,u}\Gamma(cd)
\quad \mathrm{and} \quad \Delta (c)^*\Delta(d)\ =\ M_{u,u^{1/2}}\Delta(c^*d).
\end{equation}
\end{theorem}
\begin{proof}
It is clear that \eqref{rel:Del-zero-prod} implies \eqref{eqt:disj-pres}. 
Moreover, if \eqref{rel:Del-J} and \eqref{eqt:disj-pres} hold, then the first equality of \eqref{eqt:cp-Phi-Delta} (as well as a similar one for $\Gamma$) tells us that $\Phi$ is orthogonality preserving. 
It remains to show that if $\Phi$ is orthogonality preserving, then the second statement holds. 
As $I_{F_\Phi}\subseteq I_E$, and the conclusion actually concerns with the adjointable map $\Phi_0: E\rightarrow F_\Phi$ (see Corollary \ref{cor:adj-inj}(a)), we may assume that $E$ is full.
Define $\hat\Phi: \CK(E)\to \CK(F_\Phi)$ by
$\hat\Phi (\theta) := \Phi_0\circ \theta\circ \Phi_0^*$ ($\theta\in \CK(E)$).
Since $\hat \Phi(\theta_{x,y}) = \theta_{\Phi(x),\Phi(y)}$ ($x,y \in E$), we obtain
\begin{equation}\label{eqt:rel-hat-phi}
\hat \Phi (\theta^*) = \hat
\Phi(\theta)^* \quad \mathrm{and} \quad \hat \Phi(\theta)(\Phi(z)) =
\Phi(\theta(z))u_\Phi \qquad (\theta\in \CK(E); z\in E).
\end{equation}

We define $\check \Phi: \mathfrak{L}_E \rightarrow
M(\mathfrak{L}_{F_\Phi})$ by
$\check
\Phi\Big(\begin{array}{cc}
\theta & x\\
\tilde y & a
\end{array}\Big)
:= \Big(\begin{array}{cc}
\hat\Phi(\theta) & \Phi(x)\\
\widetilde{\Phi(y)} & j_\Phi(a)
\end{array}\Big)$
(where $j_\Phi: A \to M(I_{F_\Phi})$ is the canonical map).
Then clearly,
\begin{equation}\label{eqt:check-Phi-sa}
\check\Phi(c^*)\ =\ \check \Phi(c)^* \qquad (c\in \mathfrak{L}_E).
\end{equation}
By Relations \eqref{eqt:rel-hat-phi}, it is not hard to check that 
\begin{equation}\label{eqt:rel-ch-phi}
\check \Phi (c)M_{1,u_\Phi}\check \Phi(d)\ =\ M_{u_\Phi,u_\Phi}\check\Phi(cd)
\qquad (c,d\in \mathfrak{L}_E).
\end{equation}

We may now set $\Gamma(c) :=
M_{1,u_\Phi}\check\Phi(c)$ and $\Delta(c) :=
M_{1,w_\Phi}\check\Phi(c)$ ($c,d\in
\mathfrak{L}_E$).
Observe that $\Gamma(\mathfrak{L}_E) \subseteq
\mathfrak{L}_{F_\Phi}$ because
$Au_\Phi \subseteq I_{F_\Phi}$ (by Corollary \ref{cor:adj-inj}(a)). 
On the other hand, as $w_\Phi\in C^*(u_\Phi) = \overline{C^*(u_\Phi)u_\Phi}$, one has $Aw_\Phi\subseteq \overline{Au_\Phi} =
I_{F_\Phi}$, and $\Delta(\mathfrak{L}_E) \subseteq
\mathfrak{L}_{F_\Phi}$.
It is clear that $\Gamma \circ J_E = J_F\circ \Phi = \Delta \circ J_E$.
Now, the first equality in \eqref{rel:Del-zero-prod}  follows
directly from \eqref{eqt:rel-ch-phi} and the second one follows from both \eqref{eqt:check-Phi-sa} and \eqref{eqt:rel-ch-phi}.
\end{proof}

\end{document}